\documentclass[a4paper,11pt,twoside]{article}
\usepackage[left=2.5cm, right=2.5cm,tmargin = 2.5cm]{geometry}
\usepackage{amsmath,amssymb,amsfonts,amsthm}
\usepackage{indentfirst}
\usepackage[spanish]{babel}
\usepackage[utf8]{inputenc}
\usepackage{hyperref}
\usepackage[T1]{fontenc}
\usepackage[all]{xy}
\usepackage{lipsum}
\usepackage{url}
\usepackage{stackrel}
\usepackage{graphicx}

\newtheorem{theorem}{Teorema}[section]
\newtheorem{lemma}[theorem]{Lema}
\newtheorem{example}[theorem]{Ejemplo}
\newtheorem{proposition}[theorem]{Proposición}
\newtheorem{definition}[theorem]{Definición}
\newtheorem{remark}[theorem]{Observación}

\numberwithin{equation}{section}

\begin{document}
\setcounter{page}{1}
\renewcommand{\refname}{Referencias bibliográficas}



\begin{center}
\noindent{\textbf{Introducción a la teoría de complejidad topológica}}
\vspace{0.5cm}
\end{center}
\begin{center}
\textit{{\footnotesize 
\textbf{
Cesar A. Ipanaque Zapata \footnote[1]{USP, Instituto de Ci\^{e}ncias Matem\'{a}ticas e de Computa\c{c}\~{a}o, ICMC-USP. e-mail: cesarzapata@usp.br The first author would like to thank grant\#2016/18714-8, S\~{a}o Paulo Research Foundation (FAPESP) for financial support.} y
Rodolfo José Gálvez Pérez\footnote[2]{UNMSM, Facultad de Ciencias Matemáticas, e-mail:rgalvezp@unmsm.edu.pe } }}}
\end{center}

\begin{quote}
\textbf{Resumen:} 
En este trabajo revisaremos la noción de complejidad topológica, introducida por Michael Farber en el 2003. Usaremos esta teoría de complejidad topológica para resolver el problema de planificación de movimiento de un robot móvil que navega en el plano euclidiano evitando colisionar con un obstáculo. Específicamente, calculamos la complejidad topológica y diseñamos algoritmos explícitos. 

\noindent\textbf{Palabras clave: Complejidad topológica, Problema de planificación de movimiento, Algoritmos.}
\end{quote}

\vspace{5cm}
\begin{center}
\noindent{\textbf{Introduction to the topological complexity theory}}
\end{center}
\begin{quote}
\textbf{Abstract:}  In this work we will review the notion of topological complexity, introduced by Michael Farber in 2003. We will use this theory of topological complexity to solve the motion planning problem of a mobile robot that navigates in the Euclidean plane avoiding colliding with an obstacle. Specifically, we calculate topological complexity and design explicit algorithms.

\end{quote}

\newpage
\normalsize
\section{Introducción}
El \textit{problema de planificación de movimiento} (vamos a decir simplemente MPP por sus siglas en ingles de \textit{Motion planning problem}) de robots consiste en encontrar o diseñar algoritmos que planifiquen por completo el movimiento de un sistema mecánico dado, i.e., dada una posición inicial y una posición final el algoritmo da una ruta para que el robot pueda seguir o navegar de forma autónoma desde su posición inicial hasta la posición deseada (final). Este es un problema clásico en el campo de la robótica, el lector puede ver las siguientes referencias \cite{latombe}, \cite{lavalle}, \cite{bajd}. Como ejemplo práctico para aplicar la teoría de complejidad topológica, vamos a considerar un sistema mecánico conformado por un robot móvil, que navega en el plano, y un obstáculo, como muestra la Figura~\ref{ejemplo}. 

\begin{figure}[!h]
 \caption{Sistema mecánico conformado por un robot móvil, que navega en el plano, y un obstáculo. Además, se muestra una ruta para que el robot navegue desde una posición inicial para una posición final sin chocar con el obstáculo.}
 \label{ejemplo}
\centering
 \includegraphics[scale=0.5]{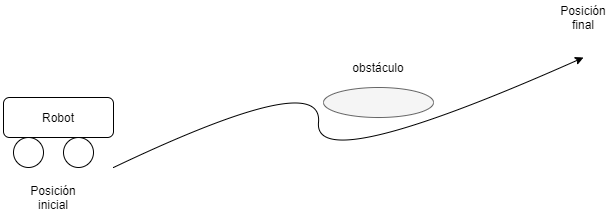}
\end{figure}

Michael Farber en el 2003, en \cite{farber}, presenta un enfoque topológico para el MPP. Sea $X$ el espacio de estados o configuraciones (libre de obstáculos) de un cierto sistema mecánico, el cual es un espacio topológico. 
Denotemos por $\text{P}X$ el espacio formado por todos los caminos continuos $\gamma:[0,1]\to X$ en $X$. $\text{P}X$ tiene la topología compacto-abierta.  Note que un camino en el espacio de configuraciones $X$ da un movimiento continuo de nuestro robot sin tocar los obstáculos, o sea, una navegación segura de nuestro robot. 

Una forma de planificar el movimiento de un sistema mecánico, y así solucionar el MPP, es encontrar un algoritmo $s$ que dé caminos en su espacio de configuraciones $X$, específicamente, dada una configuración inicial $A\in X$ y una configuración final $B\in X$, $s$ da un camino $s(A,B):[0,1]\to X$ tal que $s(A,B)(0)=A$ y $s(A,B)(1)=B$. Así $s$ da una ruta para que nuestro robot pueda navegar sin chocar con los obstáculos desde una posición inicial hasta una posición deseada (final). 

 Como $X$ es un espacio topológico entonces se puede hablar de cercanía. Si $A^\prime\in X$ está cerca de $A$ y $B^\prime\in X$ está cerca de $B$ se debe esperar que $s(A^\prime,B^\prime)$ este cerca de $s(A,B)$. En términos matemáticos se dice que $s$ dependa continuamente de las variables $A$ e $B$. En términos de robótica se dice que $s$ tenga estabilidad.
 
 Farber notó que un algoritmo de planificación de movimiento sobre $X$ es una sección de la aplicación $$e:\text{P}X\to X\times X, ~e(\gamma)=(\gamma(0),\gamma(1)).$$ En otras palabras, es una aplicación $s:X\times X\to PX$ no necesariamente continua tal que \[s(A,B)(0)=A \text{ y } s(A,B)(1)=B \text{ para todo } (A,B)\in X\times X.\] 

Para efectos de estabilidad lo conveniente es que los algoritmos sean continuos. Sin embargo, se puede mostrar que existe un algoritmo continuo sobre $X$ si, y solamente si, $X$ es contráctil. 

Por lo tanto, si $X$ no es contráctil entonces cualquier algoritmo sobre $X$ no es continuo. De esta manera, Farber define un invariante  numérico, llamado complejidad topológica $\text{TC}(X)$, que mide las discontinuidades de los algoritmos en $X$, o también, se dice que mide la complejidad del MPP en $X$.

En este trabajo vamos usar la teoría de complejidad topológica para solucionar el problema de planificación de movimiento del sistema mecánico dado en la Figura~\ref{ejemplo}, o sea, vamos a calcular su complejidad topológica y diseñar algoritmos óptimos. De esta manera, este trabajo espera  contribuir en dar una introducción a la teoría de complejidad topológica para una mayor cantidad de lectores.

\section{Complejidad topológica}
En esta sección daremos una revisión de la teoría de complejidad topológica. Específicamente responderemos las siguientes preguntas: ¿Qué es? ¿Cómo se usa? ¿Cuales son los problemas centrales? ¿Cómo se calcula? 

\subsection{¿Qué es?} Sea $X$ el espacio de estados o configuraciones (libre de obstáculo) de un cierto sistema mecánico. Denotemos por $\text{P}X$ el espacio formado por todos los caminos continuos $\gamma:[0,1]\to X$ en $X$. $\text{P}X$ es dotado de la topología compacto-abierta. Note que un camino en el espacio de configuraciones $X$ da un movimiento continuo de nuestro robot sin tocar los obstáculos, o sea, una navegación segura de nuestro robot. 

Una forma de planificar el movimiento de un sistema mecánico, y así solucionar el MPP, es encontrar un \textit{algoritmo}, digamos $s$, que dé  caminos en su espacio de configuraciones $X$, específicamente, dada una configuración inicial $A\in X$ y una configuración final $B\in X$, $s$ da un camino $s(A,B):[0,1]\to X$ tal que $s(A,B)(0)=A$ y $s(A,B)(1)=B$. Así $s$ da una ruta para que nuestro robot pueda navegar sin chocar con los obstáculos desde una posición inicial hasta una posición deseada (final). 

Note que un algoritmo en $X$ es una sección de la aplicación \begin{equation}
    e:\text{P}X\to X\times X, ~e(\gamma)=(\gamma(0),\gamma(1)).
\end{equation} En otras palabras, es una aplicación $s:X\times X\to PX$ no necesariamente continua tal que \begin{equation} s(A,B)(0)=A \text{ y } s(A,B)(1)=B \text{ para todo } (A,B)\in X\times X.\end{equation} En este caso diremos que $s$ es un \textit{algoritmo de planificación de movimiento} en $X$. Un algoritmo $s$ es continuo si la aplicación $s:X\times X\to\text{P}X$ es continua.

Recordemos que un espacio topológico $X$ es llamado \textit{contráctil} si existe una aplicación continua $H:X\times [0,1]\to X$ tal que $H_0=1_X$ e $H_1=\overline{x_0}$ para alguna constante $x_0\in X$. Donde $H_t$ denota la aplicación $H_t:X\to X$ dada por $H_t(x)=H(x,t)$, para cualquier $x\in X$,  $1_X:X\to X$ es la aplicación identidad y $\overline{x_0}$ denota la aplicación constante en $x_0$. 

Para efectos de estabilidad en los algoritmos es conveniente que los algoritmos sean continuos. Sin embargo, Farber muestra que existe un algoritmo continuo en $X$ si $X$ es contráctil. La demostración del Lema~\ref{contractil} es técnica y será usada para construir algoritmos sobre espacios contráctiles, así que presentamos una demostración conveniente para nuestros propósitos.  

\begin{lemma}\rm{\label{contractil}\cite[Theorem 1, pg. 212]{farber}
Sea $X$ un espacio topológico. Tenemos que existe un algoritmo continuo en $X$ si y solamente si $X$ es contráctil.}
\end{lemma}
\begin{proof}
$(\Rightarrow)$ Sea $s:X\times X\to \text{P}X$ un algoritmo continuo en $X$, o sea, $s$ es una aplicación continua que satisface las siguientes igualdades: $s(x_1,x_2)(0)=x_1$ y $s(x_1,x_2)(1)=x_2$, para cualquier $(x_1,x_2)\in X\times X$. Veamos que $X$ es contráctil. De hecho, fijemos un $x_0\in X$ y definamos la aplicación $H:X\times [0,1]\to X$ por \begin{equation*}
    H(x,t)=s(x,x_0)(t), \text{ para cualquier } (x,t)\in X\times [0,1].
\end{equation*} Como $s$ es continua, sigue que $H$ es continua. Además, para $x\in X$ obtenemos que \begin{eqnarray*}
H(x,0) &=& s(x,x_0)(0) \\ &=& x.
\end{eqnarray*} Similarmente, para $x\in X$ obtenemos que $H(x,1)=x_0$. Por lo tanto, $X$ es contráctil.

$(\Leftarrow)$ Sea $H:X\times [0,1]\to X$ una aplicación continua tal que $H_0=1_X$ e $H_1=\overline{x_0}$, para alguna constante $x_0\in X$. Veamos que existe un algoritmo continuo en $X$. De hecho, definamos la aplicación $s:X\times X\to \text{P}X$ dada por \[s(x_1,x_2)(t)=\begin{cases}
H(x_1,2t),& \hbox{ si $0\leq t\leq 1/2$,}\\
H(x_2,2-2t),& \hbox{ si $1/2\leq t\leq 1$,}
\end{cases} \text{ para cualquier $(x_1,x_2)\in X\times X$.}
\] Como $H$ es continua, $s$ es continua. Además, para $(x_1,x_2)\in X\times X$, tenemos que \begin{eqnarray*}
s(x_1,x_2)(0) &=& H(x_1,0)\\
&=& x_1,
\end{eqnarray*} similarmente, tenemos que $s(x_1,x_2)(1)=H(x_2,0)=x_2$. Por lo tanto, $s$ es un algoritmo continuo en $X$.
\end{proof}

Usaremos la demostración del Lema~\ref{contractil} para construir un algoritmo continuo en espacios estrellados. Recordemos que un subconjunto $K\subset\mathbb{R}^d$ es llamado \textit{estrellado} si existe un punto $x_0\in K$ tal que los puntos $(1-t)x+tx_0\in K$, para cualquier $x\in K$ y cualquier $t\in[0,1]$. En ese caso, $x_0$ es llamado la \textit{estrella} de $K$.

\begin{example}[Algoritmo continuo en espacios estrellados]
\rm{Sea $K$ un conjunto estrellado cuya estrella es $x_0\in K$. La aplicación $H:K\times [0,1]\to K$ dada por \[H(x,t)=(1-t)x+tx_0\] es una homotopia satisfaciendo $H_0=1_K$ y $H_1=\overline{x_0}$. Por lo tanto, $K$ es contráctil. Además, por la demostración del Lema~\ref{contractil} obtenemos que la aplicación $s:K\times K\to \text{P}K$ dada por \[s(x_1,x_2)(t)=\begin{cases}
(1-2t)x_1+2tx_0,& \hbox{ si $0\leq t\leq 1/2$,}\\
(2t-1)x_2+(2-2t)x_0,& \hbox{ si $1/2\leq t\leq 1$,}
\end{cases} \text{ para cualquier $(x_1,x_2)\in K\times K$,}
\] es un algoritmo continuo en $K$.}
\end{example}

El Lema~\ref{contractil} implica que si $X$ no es contráctil entonces cualquier algoritmo en $X$ no es continuo. De esta manera, Farber define un invariante  numérico, llamado complejidad topológica $\text{TC}(X)$, que mide las discontinuidades de los algoritmos en $X$, o también se dice que mide la complejidad del MPP de un sistema mecánico cuyo espacio de estados es $X$.

\begin{definition}\rm{\label{definicion-tc}\cite[Definition 2, pg. 213]{farber}
La \textit{complejidad topológica} de un espacio topológico $X$, denotado por $\text{TC}(X)$, es el menor entero positivo $m$ tal que el producto cartesiano $X\times X$ puede ser cubierto por $m$ subconjuntos abiertos $U_i$, \begin{equation*}
        X \times X = U_1 \cup U_2 \cup\cdots \cup U_m,  
    \end{equation*} tal que para cada $i = 1, 2,\ldots, m$, existe una aplicación continua $s_i:U_i\to\text{P}X$ tal que $s_i(x_1,x_2)(0)=x_1$ y $s_i(x_1,x_2)(1)=x_2$ para cualquier $(x_1,x_2)\in U_i$. En ese caso, $s_i$ es llamado un \textit{algoritmo local} continuo sobre $U_i$. Se tal $m$ no existe diremos que $\text{TC}(X)=\infty$.}  
\end{definition}

Note que $\text{TC}(X)=1$ si y solamente si existe un algoritmo continuo en $X$. Usando el Lema~\ref{contractil}, obtenemos que es posible planificar, mediante un algoritmo estable, el movimiento de un sistema mecánico si y solamente si su espacio de estados $X$ es contráctil. Así, $\text{TC}(X)=1$ si y solamente si $X$ es contráctil. Esto explica, el porqué en las aplicaciones industriales no es posible encontrar un algoritmo estable que planifique el movimiento, ya que los espacios de estados de sistemas mecánicos que mayormente aparecen en la industria no son contráctiles.

Notemos que, toda colección $s=\{s_i:U_i\to \text{P}X\}_{i=1}^k$ con cada $U_i\subset X\times X$ abierto, $X\times X=\bigcup_{i=1}^{k}U_i$ y $e\circ s_i=incl_{U_i}$, donde $incl_{U_i}:U_i\hookrightarrow X\times X$ denota la aplicación inclusión, define un algoritmo $s:X\times X\to\text{P}X$ en $X$. De hecho, para cada $(x_1,x_2)\in X\times X$ elegimos el menor $i\in \{1,\ldots,k\}$ tal que $(x_1,x_2)\in U_i$ y definimos el camino $s(x_1,x_2)$ por $s_i(x_1,x_2)$. Un tal algoritmo $s=\{s_i:U_i\to PX\}_{i=1}^k$ es llamado \textit{óptimo} se $k=\text{TC}(X)$.

Por tanto, la complejidad topológica $\text{TC}(X)$ es un invariante numérico que mide la complejidad del MPP para un sistema mecánico cuyo espacio de estados es $X$. En otras palabras, la complejidad topológica $\text{TC}(X)$ da la menor cantidad de algoritmos locales continuos que se necesitan para planificar el movimiento de un sistema mecánico cuyo espacio de estados es $X$.  

\subsection{¿Cómo se usa?} Consideremos un sistema mecánico cuyo espacio de estados es $X$. Si su complejidad topológica $\text{TC}(X)=k$ sabemos que  existe un algoritmo óptimo $s=\{s_i:U_i\to \text{P}X\}_{i=1}^k$ en $X$. Así, si damos un estado inicial $x_1$ y un estado final $x_2$, el algoritmo da una  ruta o un movimiento continuo $s(x_1,x_2)$, empezando en $x_1$ y terminando en $x_2$, para que el robot navegue de forma segura y autónoma. De esta manera, con la teoría de complejidad topológica, no únicamente planificamos el movimiento, sino también planificamos el movimiento de manera óptima, o sea, con la menor cantidad de algoritmos locales. En particular, el invariante $\text{TC}(X)$ da la menor cantidad de algoritmos locales que debe tener cualquier algoritmo en $X$. 

\subsection{¿Cuales son los problemas centrales?}
Calcular la complejidad topológica $\text{TC}$ es un problema difícil y propio de la topología algebraica. Mayormente, se desarrollan cotas inferiores y superiores que permitan estimar el valor de la TC (vea la Sección~\ref{como-calcula}). Sin embargo, esas técnicas no permiten construir o diseñar los algoritmos. Así que, podemos decir que los problemas centrales de la teoría de complejidad topológica son los siguientes:
    \begin{itemize}
        \item Calcular $\text{TC}(-)$.
        \item Encontrar algoritmos óptimos.
    \end{itemize}
 
 El problema de calcular la TC es muy estudiado en el área de la topología algebraica. Las herramientas principales usadas para estimar la TC son: Categoría de Lstuernik-Schnirelmann, teorías de cohomología, teoría de homotopía, complejos CW, fibraciones y cofibraciones, teoria de homotopía racional, etc. El lector puede consultar las siguientes referencias \cite{farber}, \cite{rudyak2010higher}, \cite{cohen2018}, \cite{cohen2020para}, \cite{cohen2020}, \cite{gonzalezhopf2019} y las referencias de ellos.
 
 El problema de encontrar o diseñar algoritmos óptimos es propio de la topología y geometría. Este es un problema poco estudiado y en general independiente del problema de calcular la TC. El lector puede consultar los trabajos \cite{zapataseq2020}, \cite{zapatamulti2020} y las referencias de ellos. 
 
 \begin{remark}
 \rm{Un problema extra y útil para fines prácticos seria la implementación, en algún programa computacional, de los algoritmos obtenidos vía la teoría de la complejidad topológica. Sin embargo, este problema aún no ha sido estudiado.}
 \end{remark}

\subsection{¿Cómo se calcula?}\label{como-calcula}  
Calcular la TC es un problema difícil y propio de la topología algebraica. Una propiedad muy útil es que la TC es un invariante homotópico. Presentamos la demostración de este hecho, ya que es muy útil y será usada en la construcción de algoritmos.

Recordemos que dos espacios topológicos $X$ y $Y$ \textit{tienen el mismo tipo de homotopía} si existen aplicaciones continuas $f:X\to Y$ y $g:Y\to X$ tales que $f\circ g\simeq 1_Y$ y $g\circ f\simeq 1_X$.

\begin{proposition}\label{homotopy-invariant}
\rm{Si existen aplicaciones continuas $Y\stackrel{g}{\to}X\stackrel{f}{\to} Y$ tales que $f\circ g\simeq 1_Y$, entonces sus complejidades topológicas satisfacen la siguiente desigualdad:  \[\text{TC}(X)\geq\text{TC}(Y).\] En particular, si $X$ y $Y$ tienen el mismo tipo de homotopía, entonces $\text{TC}(X)=\text{TC}(Y)$.}
\end{proposition}
\begin{proof}
Sea $H:Y\times [0,1]\to Y$ una homotopía tal que $H_0=1_Y$ y $H_1=f\circ g$. Note que, todo algoritmo local $s:U\to \text{P}X$ definido sobre $U\subset X\times X$ induce un algoritmo local $\widetilde{s}:V\to\text{P}Y$, definido sobre $V=(g\times g)^{-1}(U)\subset Y\times Y$, dado por:
\begin{equation}\label{induz}
   \widetilde{s}(y_1,y_2)(t) = \begin{cases}
H_{3t}(y_1),& \hbox{ si $0\leq t\leq 1/3$;}\\
f\left(s(g(y_1),g(y_2))(3t-1)\right),& \hbox{ si $1/3\leq t\leq 2/3$;}\\
H_{3-3t}(y_2),& \hbox{ si $2/3\leq t\leq 1$.}
\end{cases} 
\end{equation}
Por lo tanto, si $s=\{s_i:U_i\to \text{P}X\}_{i=1}^k$ es un algoritmo óptimo en $X$, entonces aplicando~(\ref{induz}) a cada $s_i$, podemos construir un algoritmo (no necesariamente óptimo) $\widetilde{s}=\{\widetilde{s_i}:V_i\to \text{P}Y\}_{i=1}^k$ en $Y$. Así, obtenemos que $\text{TC}(X)=k\geq\text{TC}(Y)$.  

En particular, si $X$ y $Y$ tienen el mismo tipo de homotopía, podemos obtener de forma similar la otra desigualdad y así $\text{TC}(X)=\text{TC}(Y)$. Note que, en este caso el algoritmo $\widetilde{s}$ es óptimo.
\end{proof}

Para calcular la TC, en general se intenta estimar, o sea, encontrar cotas inferiores y superiores para la TC. Las cotas inferiores mayormente dependen de la información algebraica del espacio, por ejemplo, de su anillo de cohomología singular $H^\ast(X)$ ya sea con coeficientes enteros $\mathbb{Z}$ o sobre un cuerpo $\mathbb{K}$. Las cotas superiores aparecen de la información homotópica del espacio, por ejemplo, de su dimensión homotópica y de su grado de conexidad. Continuación enunciaremos las cotas más usadas para estimar la TC.

Recordemos primero, de \cite{cornea2003}, que la \textit{categoría de Lusternik-Schnirelmann} de un espacio topológico $X$, denotado por $\text{cat}(X)$, es el menor entero positivo $k$ tal que $X$ puede ser cubierto por subconjuntos abiertos $U_1,\ldots,U_k$, o sea, $X=\bigcup_{i=1}^k U_i$, tales que cada inclusión $U_i\hookrightarrow X$ sea homotópica a una aplicación constante. Note que, $\text{cat}(X)=1$ si y solamente si $X$ es contráctil.

Para un anillo $A$ y $S\subset A$ un subconjunto, el \textit{índice de nilpotencia} de $S$ en $A$ es definido por: \[\text{Nil}(S)=\min\{k:~\text{ cualquier producto de $k$ elementos de $S$ es nulo}\}.\]

Para un complejo CW $X$, su \textit{dimensión homotópica} es dado por \[\text{hdim}(X)=\min\{\text{dim}(Y):~Y \text{ es un complejo CW y } Y\simeq X\}.\] Un espacio topológico $X$ es \textit{$q$-conexo} si $\pi_i(X)=0$ para cualquier $0\leq i\leq q$.

\begin{theorem}\label{teo} \rm{
\begin{enumerate}
    \item[1)] Para cualquier espacio conexo por caminos $X$, se tiene:
\[\text{cat}(X)\leq \text{TC}(X)\leq\text{cat}(X\times X).\]
\item[2)] Si $X$ y $Y$ son complejos CW, entonces \begin{eqnarray*}
\text{cat}(X\times Y) &\leq& \text{cat}(X)+\text{cat}(Y)-1,\\
\text{TC}(X\times Y) &\leq& \text{TC}(X)+\text{TC}(Y)-1.
\end{eqnarray*}
\item[3)] Si $X$ es un complejo CW y $q$-conexo, entonces \begin{eqnarray}
\nonumber \text{cat}(X) &\leq& \dfrac{\text{hdim}(X)}{q+1}+1,\\
\label{tc-hdim}\text{TC}(X) &\leq& \dfrac{2\text{hdim}(X)}{q+1}+1.
\end{eqnarray} 
\item[4)] Si $h^\ast$ es cualquier teoría de cohomología multiplicativa sobre los pares de espacios topológicos, entonces
\begin{eqnarray}
\label{conta-inferior} \text{Nil}\left(\text{Ker}(\Delta^\ast:h^\ast(X\times X)\to h^\ast(X)) \right) &\leq& \text{TC}(X).
\end{eqnarray}
 Donde $\Delta:X\to X\times X,~\Delta(x)=(x,x)$ es la aplicación diagonal.
\end{enumerate}}
\end{theorem}

\begin{example}[Complejidad topológica de las esferas impares]\label{spheres-odd}
\rm{Consideremos la esfera $(2m-1)$-dimensional $S^{2m-1}=\{x\in\mathbb{R}^{2m}:~\parallel x\parallel=1\}$, con $m\geq 1$, donde $\parallel \cdot\parallel$ denota la norma euclidiana. Veamos que su complejidad topológica es igual a $2$, o sea, \[\text{TC}(S^{2m-1})=2, \text{ para cualquier $m\geq 1$}.\] De hecho, como $S^{2m-1}$ no es contráctil, entonces tenemos que $\text{TC}(S^{2m-1})\geq 2$. Así, basta mostrar que $\text{TC}(S^{2m-1})\leq 2$. Para ello, vamos construir un algoritmo $s=\{s_i:U_i\to \text{P}S^{2m-1}\}_{i=1}^{2}$ en $S^{2m-1}$ de la siguiente manera:
\begin{itemize}
    \item Consideremos un campo vectorial tangente sobre la esfera $S^{2m-1}$, por ejemplo, el campo $\nu:S^{2m-1}\to S^{2m-1}$ dado por $\nu(x_1,y_1,\ldots,x_m,y_m)=(-y_1,x_1,\ldots,-y_m,x_m)$.
    \item Los abiertos $U_i$ son dados por:
    \begin{eqnarray*}
    U_1 &=& \{(a,b)\in S^{2m-1}\times S^{2m-1}:~a\neq -b\},\\
     U_2 &=& \{(a,b)\in S^{2m-1}\times S^{2m-1}:~a\neq b\}.
    \end{eqnarray*} Note que, $U_1\cup U_2=S^{2m-1}\times S^{2m-1}$.
    \item El algoritmo local $s_1:U_1\to \text{P}S^{2m-1}$ es dado por: $$ 
s_1(a,b)(t) = \dfrac{(1-t)a+tb}{\parallel (1-t)a+tb \parallel}, \text{ para todo } (a,b)\in U_1 \text{ y } t\in[0,1].
$$ Note que, fijados $(a,b)\in U_1$, o sea, $a$ y $b$ no son antípodas, el camino $s_1(a,b)(-)$ es la geodésica más corta en la esfera $S^{2m-1}$ que conecta $a$ y $b$.
\item Antes de definir el segundo algoritmo local $s_2$ consideremos el subconjunto $F=\{(a,b)\in S^{2m-1}\times S^{2m-1}: ~~a= -b\} $ y para cualquier  $(a,b)\in F$ definamos 
$$
\alpha(a,b)(t) =  \begin{cases} 
s_1(a,v(a))(2t), &\hbox{ si $0\leq t\leq 1/2$;}\\
s_1(v(a),b)(2t-1), &\hbox{ si $1/2\leq t\leq 1$.}
\end{cases} $$ Note que, $\nu(a)$ no es antípoda de $a$ ni de $b$, pues $\nu(a)$ es ortogonal a $a$ y a $b=-a$.
\item Ahora, definamos $s_2:U_2\to \text{P}S^{2m-1}$. Para cualquier,  $(a,b)\in U_2$, sea
$$ 
s_2(a,b)(t) =  \begin{cases} 
s_1(a,-b)(2t), &\hbox{ si $0\leq t\leq 1/2$;}\\
\alpha(-b,b)(2t-1), &\hbox{ si $1/2\leq t\leq 1$.}
\end{cases} $$ Note que $-b$ no es antípoda de $a$, pues $a\neq b$.
\end{itemize}
  } Así, tenemos que $\text{TC}(S^{2m-1})=2$. Además, tal algoritmo $s=\{s_i:U_i\to \text{P}S^{2m-1}\}_{i=1}^{2}$ es óptimo. \flushright$\square$
\end{example}

\begin{remark}
 \rm{Note que, usando la desigualdad~(\ref{tc-hdim}) del Teorema~\ref{teo} obtenemos la siguiente cota superior: $\text{TC}(S^{d})\leq 3$, para cualquier $d\geq 1$, pues $\text{hdim}(S^{d})=d$ y la esfera $S^d$ es $(d-1)$-conexa. Queremos resaltar que para calcular la TC de la esfera $S^{2m-1}$ hemos construido explícitamente un algoritmo y así hemos obtenido una mejor cota superior,  $\text{TC}(S^{2m-1})\leq 2$. 
 
 Por otro lado, usando el anillo de cohomología singular $H^\ast(S^d;\mathbb{Z})=\dfrac{\mathbb{Z}[\alpha]}{\langle \alpha^2\rangle}$ en la desigualdad~(\ref{conta-inferior}) del Teorema~\ref{teo} se puede obtener que $3\leq \text{TC}(S^{d})$, para $d$ par. Así,  $\text{TC}(S^{d})=3$, para cualquier $d$ par.}
\end{remark}


\section{Aplicación}
En esta sección vamos usar la teoría de complejidad topológica para solucionar el problema de planificación de movimiento para el sistema mecánico dado en la Figura~\ref{ejemplo}, o sea, vamos calcular la complejidad topológica y diseñar algoritmos.

Consideremos que nuestro robot móvil tiene un radio $l_R>0$ y que el obstáculo tiene un radio $l_O>0$. Además, consideremos que nuestro sistema de referencia $\mathbb{R}^2$ tiene su origen de coordenadas en el centro del obstáculo (vea la Figura~\ref{path-detal}). Para encontrar el espacio de estados (libre de obstáculos) de nuestro sistema mecánico, escojamos un punto de nuestro robot, por ejemplo, su centro de gravedad proyectado verticalmente sobre el plano y denotemos por $p_R$  (como muestra la Figura~\ref{path-detal}). Note que, si $p_R\in\mathbb{R}^2$ es tal que $\parallel p_R\parallel>l_O+l_R$ entonces para tal $p_R$ nuestro robot está en una posición sin tocar el obstáculo (vea la Figura~\ref{path-detal}). 

\begin{figure}[!h]
 \caption{Sistema mecánico conformado por un robot móvil de radio $l_R>0$, que navega en el plano $\mathbb{R}^2$, y un obstáculo de radio $l_O>0$. Si $p_R\in\mathbb{R}^2$ es tal que $\parallel p_R\parallel>l_O+l_R$ entonces para tal $p_R$ nuestro robot está en una posición sin tocar el obstáculo.}
 \label{path-detal}
\centering
 \includegraphics[scale=0.5]{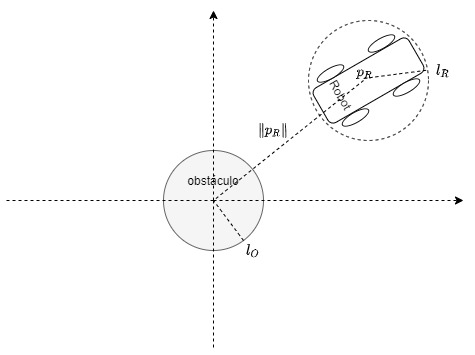}
\end{figure}

Así, tal $p_R$ determina un estado (libre de obstáculos) de nuestro robot, o sea, el espacio de estados $X$ asociado a nuestro sistema mecánico (vea la Figura~\ref{state-space}) está dado por: \[X=\{p_R\in \mathbb{R}^2:~\parallel p_R\parallel>l_O+l_R\}.\] 

\begin{figure}[!h]
 \caption{El espacio de estados $X$.}
 \label{state-space}
\centering
 \includegraphics[scale=0.5]{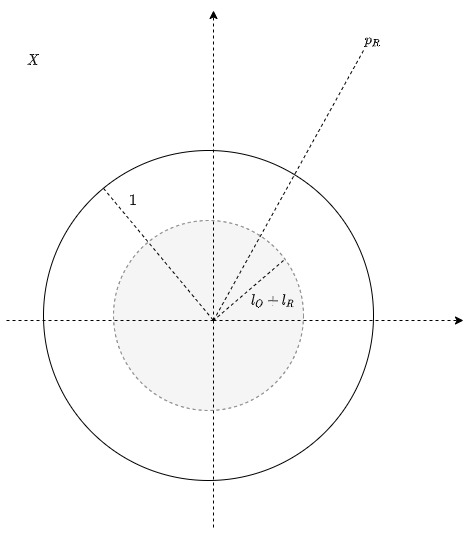}
\end{figure}

Sin perdida de generalidad, podemos suponer que $l_O+l_R<1$ (como muestra la Figura~\ref{state-space}). Luego, podemos mostrar que $X$ \textit{se retrae por deformación sobre} $S^1$, o sea, existe una homotopía $H:X\times [0,1]\to X$ satisfaciendo $H_0=1_X$, $H_1(X)\subset S^1$ y $H_1(z)=z$, para cualquier $z\in S^1$. Considerando $r=H_1:X\to S^1$ y $i:S^1\hookrightarrow X$ la aplicación inclusión, tenemos que $r\circ i=1_{S^1}$ y $i\circ r\simeq 1_X$. En particular, $X$ tiene el mismo tipo de homotopía que la $1$-esfera $S^1$. De hecho, basta definir \begin{equation}
    \label{homotopia}  H(p_R,t)=(1-t)p_R+t\dfrac{p_R}{\parallel p_R\parallel},
\end{equation} para cualquier $(p_R,t)\in X\times [0,1]$.

Así, obtenemos la complejidad topológica para $X$.

\begin{proposition}
\rm{La complejidad topológica para $X$ es dada por:\[\text{TC}(X)=2.\]}
\end{proposition}
\begin{proof}
Como $X$ y $S^1$ tienen el mismo tipo de homotopía, por la Proposición~\ref{homotopy-invariant}, sigue que $\text{TC}(X)=\text{TC}(S^1)$. Usando el Ejemplo~\ref{spheres-odd}, concluimos que $\text{TC}(X)=2.$
\end{proof}

Por otro lado, como la homotopía $H$ dada en $(\ref{homotopia})$ es explicita, entonces, usando la demostración de la Proposición~\ref{homotopy-invariant}, tenemos que el algoritmo óptimo $s$ en la esfera $S^1$, dado en el Ejemplo~\ref{spheres-odd}, induce un algoritmo óptimo en $X$.

\begin{proposition}
\rm{Existe un algoritmo óptimo explícito en $X$.}
\end{proposition}
\begin{proof}
Sea $s=\{s_i:U_i\to \text{P}S^{1}\}_{i=1}^{2}$ el algoritmo óptimo en $S^{1}$, dado en el Ejemplo~\ref{spheres-odd}. Por la demostración de la Proposición~\ref{homotopy-invariant}, tenemos que $s$ induce un algoritmo óptimo $\widetilde{s}=\{\widetilde{s}_i:V_i\to \text{P} X\}_{i=1}^{2}$ en $X$. Donde \begin{eqnarray*}
V_1 &=& (r\times r)^{-1}(U_1)\\
&=& \{(z_1,z_2)\in X\times X:~\left(r(z_1),r(z_2)\right)\in U_1\}\\
&=& \{(z_1,z_2)\in X\times X:~ \dfrac{z_1}{\parallel z_1\parallel}\neq -\dfrac{z_2}{\parallel z_2\parallel}\}.
\end{eqnarray*} Similarmente, tenemos que \begin{eqnarray*}
V_2 &=& \{(z_1,z_2)\in X\times X:~ \dfrac{z_1}{\parallel z_1\parallel}\neq \dfrac{z_2}{\parallel z_2\parallel}\}.
\end{eqnarray*} Los algoritmos locales $\widetilde{s}_i$ están dados por:
\begin{eqnarray*}
\widetilde{s}_i(z_1,z_2)(t) &=&  \begin{cases}
H_{3t}(z_1),& \hbox{ si $0\leq t\leq 1/3$;}\\
i\left(s_i(r(z_1),r(z_2))(3t-1)\right),& \hbox{ si $1/3\leq t\leq 2/3$;}\\
H_{3-3t}(z_2),& \hbox{ si $2/3\leq t\leq 1$.}
\end{cases} 
\end{eqnarray*} 

\begin{figure}[!h]
\begin{minipage}[b]{0.5\textwidth}
\begin{center}
\includegraphics[scale=0.4]{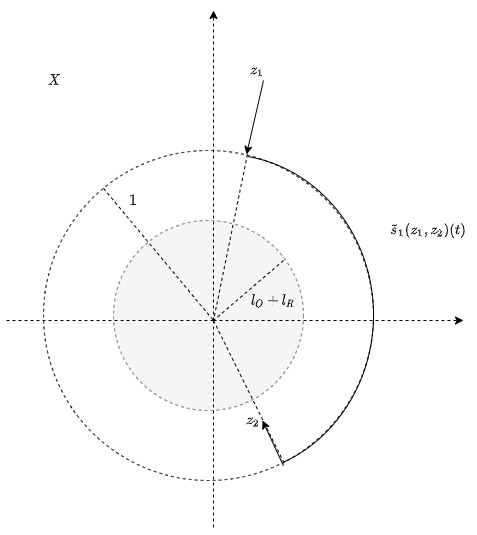}
\caption{El algoritmo local $\widetilde{s}_1$.}
\label{fig-1}
\end{center}
\end{minipage} \hfill \begin{minipage}[b]{0.5\textwidth}
\begin{center}
\includegraphics[scale=0.4]{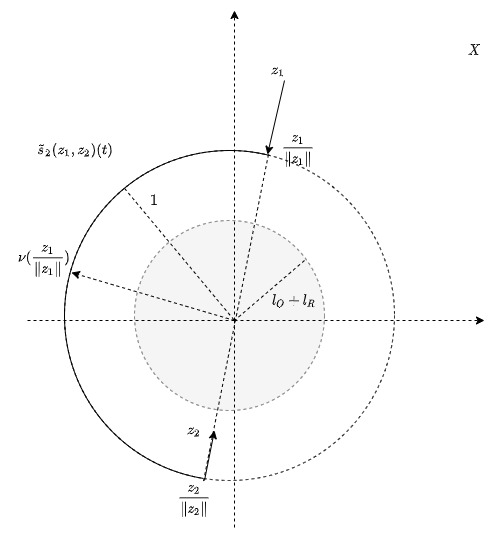}
\caption{El algoritmo local $\widetilde{s}_2$.}
\end{center}
\end{minipage}
\end{figure}

\end{proof}

Finalmente, en la siguiente Observación, vamos usar el algoritmo $\widetilde{s}$ para planificar el movimiento de nuestro robot, sin colisionar con el obstáculo, desde una posición inicial hasta una posición deseada. Así, $\widetilde{s}$ torna a nuestro robot en un robot autónomo capaz de navegar en el plano sin colisionar con el obstáculo. 

\begin{remark}
\rm{Note que el camino  $\widetilde{s}_1(z_1,z_2)(t)$ dado en la Figura~\ref{fig-1} induce un movimiento continuo en nuestro sistema mecánico como muestra la Figura~\ref{moc-cont}.} 

\begin{figure}[!h]
 \caption{Movimiento continuo del robot sin colisionar con el obstáculo desde la posición inicial $p_R=z_1$ hacia la posición final $p_R=z_2$ usando la ruta $\widetilde{s}_1(z_1,z_2)(t)$.}
 \label{moc-cont}
\centering
 \includegraphics[scale=0.5]{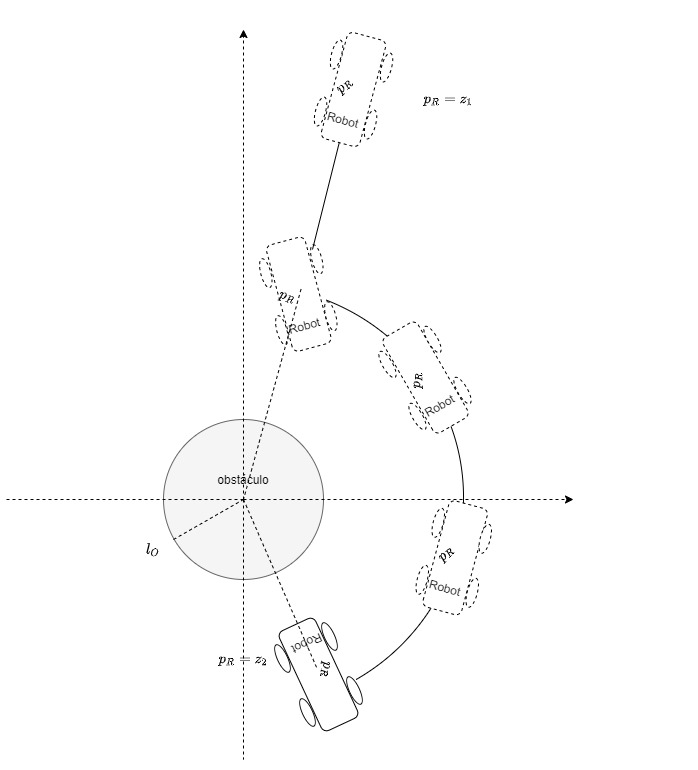}
\end{figure}
\end{remark}

\section{Conclusión}
La teoría de complejidad topológica es un enfoque topológico desarrollado por Farber en el 2003 para solucionar el problema de planificación de movimiento de robots. Hemos dado una introducción básica a esta teoría y para resaltar su importancia hemos resuelto el problema de planificación de movimiento de un sistema mecánico que consiste de un robot móvil que navega en el plano sin colisionar con un obstáculo. Así, este trabajo espera ser una introducción básica a la teoría de complejidad topológica para una mayor cantidad de lectores.

\end{document}